\documentclass[11pt]{article}
\usepackage{mathtools}
\usepackage[pagewise]{lineno}%\linenumbers

\usepackage{hyperref}
\usepackage{amsmath,amssymb,amstext,amsfonts}
\usepackage{algorithm}
\usepackage{algorithmic}
\usepackage{relsize}
\usepackage{amsthm}
\usepackage{cite}
\usepackage{enumerate}
\usepackage[all,arc]{xy}
\usepackage{mathrsfs}
\usepackage{tcolorbox}
\usepackage{graphicx}
\usepackage{subcaption}
\usepackage{float}
\usepackage[margin=0.9 in]{geometry}
\usepackage{listings}
\usepackage{xcolor}
\usepackage{subcaption}

\usepackage{graphicx}
\usepackage[utf8]{inputenc}
\newtheorem{theorem}{Theorem}[section]
\newtheorem{lemma}[theorem]{Lemma}
\newtheorem{proposition}[theorem]{Proposition}

\newtheorem{definition}{Definition}[section]
\newtheorem{assumption}{Assumption}
\theoremstyle{definition}
\newtheorem{example}{Example}[section]
\usepackage{systeme}
\usepackage{comment}
\usepackage{dsfont}

\tcbuselibrary{theorems}
\theoremstyle{definition}

\newcounter{texercise}
\newwrite\solout
\def\openoutsol{\immediate\openout\solout\jobname.sol}  \def\writesol#1{\immediate\write\solout{\noexpand\processsol{\thetexercise}{#1}}}% \def\closeoutsol{\immediate\closeout\solout} \def\inputsol{\IfFileExists{\jobname.sol}{\input{\jobname.sol}}{}}

\newcounter{mytheorem}[section] 
\tcbset{
theostyle/.style={fonttitle=\bfseries\upshape, fontupper=\slshape,
arc=0mm, colback=blue!5,colframe=blue!75!black}, defstyle/.style={fonttitle=\bfseries\upshape, fontupper=\slshape,
colback=red!10,colframe=red!75!black},
}

\newcommand{\calf}{\mathcal{F}}
\newcommand{\bR}{\mathbb{R}}

\newcommand{\E}{\mathbb{E}}

\numberwithin{equation}{section}
\date{}

\begin{document}
\title{Solving high dimensional FBSDE with deep signature techniques with application to nonlinear options pricing}
\author{Hui Sun\thanks{Citibank, D.E. USA. Disclaimer: the content of the paper is of the author's own research interest and does not represent any of the corporate opinion.}\ , \ Feng bao \thanks{Department of Mathematics, Florida State University}}
\maketitle
\begin{abstract}
We report two methods for solving FBSDEs of path dependent types of high dimensions. Inspired by the work from \cite{haojie}, \cite{weinan1} and \cite{lyons1}, we propose a deep learning framework for solving such problems using path signatures as underlying features. Our two methods (forward/backward) demonstrate comparable/better accuracy and efficiency compared to the state of the art \cite{qifeng1}, \cite{qifeng2} and \cite{come1}. More importantly, leveraging the techniques developed in \cite{lyons5} we are able to solve the problem of high dimension which is an limitation in \cite{qifeng1} and \cite{qifeng2}. We also provide convergence proof for both methods with the proof of the backward methods inspired from \cite{ruimeng} in the Markovian case. 
\end{abstract}
%\tableofcontents
%\listoffigures
%\listoftables
\section{Introduction}
Backward stochastic differential equations (BSDEs) and numerical solutions for solving BSDEs have been extensively studied in the past few decades \cite{Bao_first, BCZ_2011, BCZ_2015, BCZ_2018, Jianfeng, Zhao_BSDE}.
Recently, solving high dimensional PDEs through BSDEs and machine learning attracted a lot of attentions and researchers invested a large amount of effort in designing numerical schemes to solve such problems (c.f.\cite{weinan1, fin1, fin7, come1, BSDE_filter}). The workhorse of almost all such schemes is the combination of usage of BSDEs/FBSDEs and and deep neural networks. The BSDEs/FBSDEs provide a probabilistic interpretations for the PDEs of interest which makes simulation method possible, while deep learning techniques provide the uniform approximation power which most importantly is not sensitive to dimensions. 

The classical decoupled forward-backward SDE takes the following form
\begin{align} \label{fbsde1}
	\begin{cases}
		dX_t = b(t,X_t) dt + \sigma (t,X_t)dW_t \\
		dY_t= -f(t,X_t,Y_t,Z_t) dt + Z_t dW_t \\ 
		X_0=x, Y_T=g(X_T)
	\end{cases}
\end{align}
where  $(b,\sigma): [0,T] \times \bR^{d_1} \rightarrow \bR^{d_1}\times \bR^{d_1 \times d}$  and $f$ is the driver of the BSDE:
$ f: [0,T] \times  \bR^{d_1} \times  \bR^{d_2} \times \bR^{d_2 \times d} \rightarrow \bR^{d_2}$ and $g: \bR^{d_1} \rightarrow \bR^{d_2}$. When the drift and diffusion of the SDE, the driver of the BSDE and the terminal function only dependent on the current state $X_t$, we have a Markovian system. By the non-linear Feynman-Kac formula, the corresponding PDEs  are of the semi-parabolic type. 

On the other hand, of particular interest is the PDE of path dependent types (PPDE). In this case, such PPDEs are related to non-Markovian FBSDEs where the coefficients ($f, g$ etc.) in the system can depend on the entire path of the stochastic process. However, allowing the variables to be path-dependent can lead to challenges both theoretically and numerically.

On the otherhand, Solving the PPDE/FBSDE is of great interest since they arise naturally in the various financial context, see \cite{standard_1},\cite{standard_2},\cite{standard_3},\cite{standard_4}. Recently, some attentions are given to finding the numerical solutions of the PPDEs \cite{ppde1} where LSTM together with path signatures are used and \cite{ppde2} where the LSTM method is used with the deep Galerkin method. On the other hand, we note that the algorithm in \cite{ppde2} takes a long time to converge.  In the meantime, some option pricing problems arising from the Volterra SDEs also lead to path-dependent PDEs. Numerical algorithms are considered to solve such problems \cite{ppde3}, \cite{ppde4}.

In this work, inspired by the work in \cite{lyons1} \cite{lyons2} and the theoretical findings in \cite{lyons3}, we propose to introduce the path signature as the underlying features in replacement of the original paths of the SDE. The benefit, as pointed out in both \cite{lyons1} and \cite{qifeng1}, is that financial timeseries data is usually of high frequencies and path signature does not lose information when truncated at discrete time. That being said, no down sampling is needed for time series with large number of steps.

We noted that in \cite{lyons1}, it is reported the RNN structure together with the path signatures can very well learn the solution of an SDE given only the simulated Brownian paths. We observe in the literature \cite{qifeng1} that the exact same idea was adopted to solve non-markovian FBSDEs. More specifically, SDE paths are first simulated on a fine meshgrid. Then truncated path signatures/log-signatures are generated on segments of the fine meshgrid which are defined through a coarse meshgrid. Then signatures on each segments will be used as features. Together with the RNN structure, the function $Z_{t_n}$ in the BSDE will be approximated recursively at discrete time locations.  

In this work, we report two new methods (forward and backward methods) which are also based on using the path signature similar to \cite{qifeng1} and \cite{qifeng2} as features. The contributions of this paper can be summarized as follows:  
\begin{itemize}
    \item Two neural network structures are proposed with proof of convergence  given and numerical examples provided. We use the framework same as \cite{weinan1}. That is, instead of using RNNs for the function approximation purpose, we use individual neural networks to learn the function $Z_{t_i}$ at each discrete time, where the argument of the function will be the signature of the entire path $X_t$ up to time $t_i$. Since individual neural networks are used, the input will be the truncated signatures of paths staring from time 0. 
    \item The method in \cite{qifeng1} can only be used to solve FBSDEs related to semi-linear parabolic PDEs of the path dependent type. We also propose a backward algorithm that can solve the reflected FSBDE (optimal stopping types) which are directly related to the pricing of American options. The algorithm is inspired by the work in \cite{haojie}. 
   \item  We solve the open computational problem stated in \cite{qifeng1}. Both the work in \cite{qifeng1} and \cite{qifeng2} has the limitation of not being able to solve problems of high dimensions: the highest dimension problem they can solve is of 20 dimension. We propose a method that uses the technique from \cite{lyons5}. More specifically, we add an embedding layer prior to finding the signature of a path. The embedding layer plays the role of dimension reduction and information extraction. We remark that the data-stream structure is still preserved after the original data stream is processed.  
\end{itemize}

The result of the paper is structured as follows. In section 2 we give a brief introduction of the signature method and state our main algorithms for both the forward and backward methodologies. We comment that a different backward algorithm was also considered by \cite{qifeng2}.
Assumptions made in the paper are given in section 3. Convergence analysis for both the markovian and non-markovian FBSDEs are provided in Section 4. The proof for the backward algorithm is similar to that in \cite{ruimeng} with minor changes, we provide it here for completeness. Lastly, in section 5 we provide numerical examples and compare our results to those in the literature. The code for all the test can be found in Github \url{https://github.com/Huisun317/path-dependent-FBSDE-/tree/main}

\section{Algorithm and Notations}
\subsection{A quick introduction to Signature methods}
In this section, we give a brief introduction to the path signature. We first provide definition of the signature of a path and then present a few examples which is then followed by the log-signature. We comment that the dimension of signature of a path increase exponentially with the dimension of the underlying state process. Log-signature on the other hand, will effectively reduce the dimension without losing the information content of the path signature. We then introduce the Chen's identity which is needed for the efficiency of numerical computation of signatures. Lastly, we introduce the Universal nonlinearity of proposition which shows the approximation power of functions using signature as features to the path dependent functions.  

\begin{definition}
    \textbf{(Definition A.1.\cite{lyons5})} Let $a,b \in \bR$ and $X=(X^1,..., X^d):[a,b] \rightarrow \bR^d$ be a continuous piecewise smooth path. The signature of $X$ is then defined as the collection of iterated integrals. 
    \begin{align}
        Sig(X)_{a,b}=&\Big(\underset{{a <t_1 < ...<t_k < b}}{\int...\int} dX_{t_1} \otimes ...\otimes dX_{t_k} \Big)_{k \geq 0} \nonumber \\ 
        & =\Bigg( \Big( \underset{{a <t_1 < ...<t_k < b}}{\int...\int} dX^{i_1}_{t_1} \otimes ...\otimes dX^{i_k}_{t_k} \Big)_{1 \leq i_1 , ..., i_k \leq d} \Bigg)_{k\geq 0}
    \end{align}
\end{definition}
We note that the truncated signature is defined as 
\begin{align}
    Sig^m(X)_{a,b}=&\Big(\underset{{a <t_1 < ...<t_k < b}}{\int...\int} dX_{t_1} \otimes ...\otimes dX_{t_k} \Big)_{0 \leq k \leq m}
\end{align}
The signature can be expressed in the formal power series
\begin{align}
    S(X)_{a,b} = \sum^{\infty}_{k=0} \sum_{i_1,...,i_k \in \lbrace 1,...,d \rbrace} S(X)^{i_1,...,i_k}_{a,b} e_{i_1}...e_{i_k}
\end{align}
We provide simple examples for illustration of path signatures. For more examples and theoretical results see \cite{lyons1} \cite{lyons2} \cite{lyons3} \cite{lyons4} and \cite{lyons5}. 
\begin{example}
      We take $d=1$ then we have 
      \begin{align}
          Sig^1(X)_{a,b}&= X_b -X_a \nonumber \\
          Sig^2(X)_{a,b}&= \frac{( X_b -X_a)^2}{2! }\nonumber \\
          Sig^3(X)_{a,b}&= \frac{( X_b -X_a)^3}{3! }\nonumber \\
          &... \nonumber
      \end{align}
Then the signature of path $X: [a,b] \rightarrow \bR^d$ is given by 
\begin{align}
    (Sig^1(X)_{a,b}, \ Sig^2(X)_{a,b}, \ Sig^3(X)_{a,b},...) \nonumber
\end{align}
\end{example}
  
\begin{example}
      We take $d=2$ and for $X_t=(X_t^1, X_t^2)$
      \begin{align}
          Sig^1(X)_{a,b}&= (\int^b_a dX^1_{t_1}, \int^b_a dX^2_{t_2} )\nonumber \\
          Sig^2(X)_{a,b}&= (\int_a^{b}\int^{t_2}_a dX^1_{t_1} dX^1_{t_2} ,\int_a^{b}\int^{t_2}_a dX^1_{t_1} dX^2_{t_2},  \int_a^{b}\int^{t_2}_a dX^2_{t_1} dX^2_{t_2},\int_a^{b}\int^{t_2}_a dX^2_{t_1} dX^1_{t_2})   \nonumber \\
          & ... 
      \end{align}
      Then the signature of path $X: [a,b] \rightarrow \bR^d$ is given by 
\begin{align}
    (Sig^1(X)_{a,b}, \ Sig^2(X)_{a,b},...) \nonumber
\end{align}
\end{example}

We comment that there is a transformation of the the path signature called the log signature which corresponds to taking the formal logarithm of the signature in the algebra of formal power series. For a power series $x$ where  
\begin{align}
     x= \sum^{\infty}_{k=0} \sum_{i_1,...,i_k \in \lbrace 1,...,d \rbrace} \lambda_{i_1,...,i_k} e_{i_1}...e_{i_k}
\end{align}
for $\lambda_0 >0$, the logarithm is defined as 
\begin{align}
    \log(x) = \log(\lambda_0) + \sum_{n \geq 1} \frac{(-1)^n}{n}(1-\frac{x}{\lambda_0})^{\otimes n}
\end{align}
\begin{definition}
    (Log signature \textbf{Definition 6 \cite{lyons4}}). For a path $X:[a,b] \rightarrow \bR^d$, the log signature of $X$ is defined as the formal power series $\log S(X)_{a,b}$. 
\end{definition}
\begin{definition}
    (Concatenation \textbf{Definition 4 \cite{lyons4}}) For two paths $X: [a,b] \rightarrow \bR^d$ and $Y: [b,c] \rightarrow \bR^d$, we define the concatenation as the path $X*Y:[a,c] \rightarrow \bR^d$ for which 
    \begin{align}
    (X*Y)_t=
        \begin{cases}
             X_t , \ \ t \in [a,b] \\ 
             X_b +(Y_t - Y_b), \ \ t \in [b,c] \\ 
        \end{cases}
    \end{align}
\end{definition}
The Chen's identity can provide a simpler computational method for two paths that are concatenated. This theorem will be helpful if one needs to consider signature of different paths that have the same origin and large amount of overlaps. 
\begin{theorem}
    (Chen's Identity \textbf{Theorem 2 \cite{lyons4}}). Let $X: [a,b] \rightarrow \bR^d$ and $Y: [b,c] \rightarrow \bR^d$ be two paths. Then 
    \begin{align}
        S(X*Y)_{a,c} = S(X)_{a,b}\otimes S(Y)_{b,c}
    \end{align}
\end{theorem}

We also give the definition of the time augmented path. This is needed because in the numerical implementation, we will always use this augmented process. 
\begin{definition}
    Given a path $X:[a,b] \rightarrow \bR^d$, we define the corresponding time-augmented path by $\hat{X}=(t,X_t)$ which is a path in $\bR^{d+1}$. 
\end{definition}
We provide the following proposition on Universal non-linearity as it gives the theoretical grounding/motivation for us to use deep neural network to approximate the functional $F$. Basically, what the proposition says is that the function of the path is approximately linear on the signature. In some sense, the signature can be treated as a `universal nonlinearity' on paths.
\begin{proposition}
(Universal nonlinearity \textbf{Proposition A.6 \cite{lyons5}})
    Let $F$ be a real-valued continuous function on continuous piecewise smooth paths in $\bR^d$ and let $\mathcal{K}$ be a compact set of such paths. Furthermore assume that $X_0$ for all $X \in \mathcal{K}$. Let $\epsilon >0$. Then there exists a linear functional $L$ such that for all $X \in \mathcal{K}$, 
    \begin{align}
        |F(X)-L(Sig(\hat{X})| < \epsilon
    \end{align}
\end{proposition}

\subsection{Algorithms}
We use the following numerical scheme to approximate the solution of the forward SDE for \eqref{fbsde1}: 
\begin{align} \label{X_tilde}
    X_{\tilde t_{i+1}}^{\tilde N} = X_{\tilde t_{i}}^{\tilde N} +b(\tilde t_{i},X_{\tilde t_{i}}^{\tilde N}) h + \sigma (\tilde t_{i},X_{\tilde t_{i}}^{\tilde N}) \Delta W^{\tilde N}_{\tilde{t}_i}
\end{align}
where use tilde to denote that we have a fine mesh, and $h=T/\tilde{N}$.
Based on the fine meshgrid, we also have a coarse mesh grid that has total $N$ segment where we define $\Delta t: = \frac{T}{N}=hM$ where $M$ is the number of fine grids in each segment. In this case, it is then clear that $N=T/hM$.  Note that $ \Delta W^{\tilde N}_{\tilde{t}_i}= W^{\tilde N}_{\tilde{t}_{i+1}}-W^{\tilde N}_{\tilde{t}_i}$. Notice that one can  define $\lbrace X^N_{t_n} \rbrace_{1 \leq n \leq N}$ by naturally taking snapshots of $\lbrace X^{\tilde N}_{\tilde t_i} \rbrace_{0\leq i \leq \tilde{N}}$ at $i=0, M, 2M, ...,MN$. $ \lbrace W^N_n \rbrace_{0 \leq n \leq N-1}$ can be defined similarly. The motivation for defining such two solutions with total steps of size $\tilde{N}$ and $N$ is that the financial 

To approximate $Z_{t}$ we use a sequence of neural networks defined on this courser meshgrid $\lbrace 0, \frac{T}{N}, ,... \frac{nT}{N}, ..., T \rbrace$ where we denote the approximator as $Z^{N, sig}_{n-1}$. Each neural network will take the truncated signature of the path $X_t$ of level $m$ as input. We denote the truncated signature of the process $X$ up until time $t$ as $\pi_{m}(Sig(X_{0: t}))$, and that of the approximated path according to \eqref{X_tilde} as
\begin{align}
    \pi_{m}(Sig(X^{\tilde N}_{0:\tilde t_{i}})),  \ \ \tilde t_{i}=t_0, t_1, .., t_n , ..., t_{N-1}
\end{align} 

In our first approach, the sample-wise solution of the BSDE is then approximated by the following discretization. 
\begin{align}\label{bsde_alg1}
    Y^{N, sig}_{n}=Y^{N, sig}_{n-1} -f_{n-1}(X^N_{n-1},Y^{N, sig}_{n-1},Z^{N, sig}_{n-1}) \Delta t + Z^{N, sig}_{n-1} \Delta W^N_{n}
\end{align}
where we use the short hand notation $f_{n}(\cdot, \cdot, \cdot):= f(t_{n},\cdot, \cdot, \cdot)$, and use $n$ in place for $t_n$.We remark that here $X^N_{n}:=X^N_{t_n}$ is essentially $X_{Mhn}^{\tilde N}$. In later analysis, we will also use $(X^N_{t_n},Y^{N, sig}_{t_n},Z^{N, sig}_{t_n})$ in place of $(X^N_{n},Y^{N, sig}_{n},Z^{N, sig}_{n})$ to make it clearly that those are the approximation at time $t_n$. This should not cause any confusion since the upper index $N$ denotes that those are quantities obtained through numerical scheme.

We did not specify the terminal condition/discretization because in our first approach, the sample $Y$ propagates forward: we initialize a batch of the $Y_0$ and propagate through \eqref{bsde_alg1} to match the terminal condition. 

On the other hand, we note that such forward method cannot deal with the optimal stopping problems which typically requires a backward scheme. As such, inspired by the algorithm in \cite{haojie}, we also propose a backward algorithm. But in this case, the process $Z_t$ takes the path of the process $X_t$. 
\begin{align}\label{bsde_alg2}
    Y^{N, sig}_{n-1}=Y^{N, sig}_{n} + f_{n-1}(X^N_{n-1},Y^{N, sig}_{n},Z^{N, sig}_{n-1}) \Delta t + Z^{N, sig}_{n-1} \Delta W^N_{n}
\end{align}
where $Z_n^{N, Sig} := Z_n^{\theta_n}(\pi_m(Sig(\tilde X_{0: n\Delta t}^{j,\tilde N})))$ and $\lbrace Z_n^{\theta_n} \rbrace_{i=1,...N-1}$. is the feed-forward neural network estimator. 

We point out that one of the motivation for the backward algorithm is it can be extended to solve the optimal stopping problem whose PDE counterpart in the Markovian case is of the following form
\begin{align}
    \begin{cases}
        \min \lbrace -\partial_t u - \mathcal{L}u - f(t, x, u, \sigma^T D_x u), u -g \rbrace, \ \ t \in [0,T), x \in \bR^d \\ 
        u(T,x)=g(x), x \in \bR^d\\
    \end{cases}
\end{align}

This variational inequality can be linked to RBSDE (Reflected BSDE).  
\begin{align}{\label{reflected_BSDE}}
    X_t & =x+\int^t_0 b(s,X_s)ds  + \int^t_0 \sigma(s,X_s) dW_s \nonumber \\ 
    Y_t&=g(X_T)+\int^T_t f(s,X_s,Y_s,Z_s) ds -\int ^T_t Z_s dW_s + K_T -K_t \nonumber \\ 
    Y_t &\geq g(X_t), \ 0\leq t \leq T
\end{align}
where $K$ is an adapted non-decreasing process satisfying 
\begin{align*}
    \int^T_0 Y_t-g(X_t) dK_t = 0
\end{align*}
Accordingly, the numerical scheme in general takes the following form
\begin{align}\label{reflected_num}
\begin{cases}
    Y^{N,Sig}_T = g(X^{\tilde N}_{0:T})  \\ 
    Y^{N, sig}_{n-1} =Y^{N, sig}_{n} + f_{n-1}(X^N_{n-1},Y^{N, sig}_{n},Z^{N, sig}_{n-1}) \Delta t - Z^{N, sig}_{n-1} \Delta W^N_{n} \\ 
    Y^{N, sig}_{n-1} =\max(g(X^N_{n-1}), Y^{N, sig}_{n-1}) 
\end{cases}
\end{align}
We state the algorithms for both the forward and backward Deep Signature algorithm in \textbf{Algorithm }\ref{algorithm 1}. 

Finally, we note that due to the fact that the dimension of the truncated signatures increase exponentially with the dimension of the underlying paths, solving path-dependent FBSDEs of very high dimensions become impractical when the state process is of very high dimension. (The generated signature if of dimension $(d^{m+1}/(d-1)$ where $d$ is the dimension of the state process and $m$ is the level of the truncated depth of the signature).

The idea is to pass the underlying paths say $\lbrace X_{t_n} \rbrace_{0 \leq n \leq N-1}$ where $X_{t_n} \in \bR^d$ through an embedding layer while still keeping the sequence structure. The embedded result then will be used to generate signatures which will be passed through a sequence of individual neural networks or RNNs \cite{qifeng1} to approximate $Z_{t_n}$. This will require error back-propagation of the signatures since the parameters of the embedded layers needs to be trained. As such, we will use the Signatory library in Pytorch since it provides such calculation \cite{lyons5}.

\begin{algorithm}
\caption{Algorithm for forward deep FBSDE Signature method}\label{algorithm 1}
\begin{algorithmic}[1]
\REQUIRE Initializing the following terms
\begin{itemize}
    \item $Y_0$ and a margin $\epsilon$, $epoch=0$ total number of Iterations Epoch. 
    \item Feedforward Neural network $\lbrace Z_n^{\theta_n} \rbrace_{i=0,...N-1}$. 
    \item Time discretization $h$ which determines the total number of temporal discretization $\tilde N$  and the total number of segment $N$ which determines coarser grid size $\Delta t$. 
\end{itemize}
\WHILE{$LOSS(Y_0) \geq \epsilon$ or $Iter > Epoch$}
    \STATE{ 
    \begin{itemize}
        \item Randomly sample batch $B$ of Brownian paths ($\tilde W_{0},\tilde W_{h},..., \tilde W_{nMh}, ... \tilde W_{N \tilde h} $) and accordingly the state process ($\tilde X^{j,\tilde N}_{0},...,\tilde X^{j,\tilde N}_{h}, \tilde X^{j,\tilde N}_{nMh}, ... \tilde X^{j,\tilde N}_{\tilde N \tilde h}$)
        \item Create truncated signature (log-signature) 
        $$\Big(\pi_m(Sig(\tilde X_{0}^{j,\tilde N})),\pi_m(Sig(\tilde X_{0:\Delta t}^{j,\tilde N})),...,\pi_m(Sig(\tilde X_{0: n\Delta t}^{j,\tilde N})),..., \pi_m(Sig(\tilde X_{0:N\Delta t}^{j,\tilde N}))\Big)_{0 \leq j \leq B}$$
        \item Compute $Z_n^{j, N, Sig} = Z_n^{\theta_n}(\pi_m(Sig(\tilde X_{0: n\Delta t}^{j,\tilde N})))$
        \item For each path $\lbrace \tilde X^{j,\tilde N}_{0:\tilde N} \rbrace_{1 \leq j \leq B}$, find $Y_n^{j,N, Sig}$ iteratively using the Euler scheme \eqref{bsde_alg1}. 
    \end{itemize}}
    \STATE{Compute the loss by matching the terminal conditions
    $$Loss(Y_0)= \frac{1}{B} \sum^B_{j=1} (Y_T^{j,N,Sig}-g(X_T^{j,\tilde N, Sig}))^2$$}
    
    \STATE{Update the parameters through gradient descent. $Iter=Iter+1$}
\ENDWHILE
\RETURN $Y_0$
\end{algorithmic}
\end{algorithm}

\begin{algorithm}
\caption{Algorithm for Backward deep FBSDE Signature method}\label{algorithm 2}
\begin{algorithmic}[2]
\REQUIRE \STATE{Initializing the following terms}
\begin{itemize}
    \item A margin $\epsilon$, $epoch=0$ total number of Iterations Epoch. 
    \item Feedforward Neural network $\lbrace Z_n^{\theta_n} \rbrace_{i=0,...N-1}$. 
    \item Time discretization $h$ which determines the total number of temporal discretization $\tilde N$  and the total number of segment $N$ which determines coarser grid size $\Delta t$. 
\end{itemize}
\WHILE{$LOSS(Y_0,Z_0) \geq \epsilon$ or $Iter > Epoch$}
    \STATE{ 
    \begin{itemize}
        \item Randomly sample batch $B$ of Brownian paths ($\tilde W_{0},\tilde W_{h},..., \tilde W_{nMh}, ... \tilde W_{\tilde N \tilde h} $) and accordingly the state process ($\tilde X^{j,\tilde N}_{0},\tilde X^{j,\tilde N}_{h},..., \tilde X^{j,\tilde N}_{nMh}, ... \tilde X^{j,\tilde N}_{\tilde N \tilde h}$)
        \item Create truncated signature (log-signature) 
        $$\Big(\pi_m(Sig(\tilde X_{0}^{j,\tilde N})),\pi_m(Sig(\tilde X_{0:\Delta t}^{j,\tilde N})),...,\pi_m(Sig(\tilde X_{0: n\Delta t}^{j,\tilde N})),..., \pi_m(Sig(\tilde X_{0:N\Delta t}^{j,\tilde N}))\Big)_{0 \leq j \leq B}$$
        \item Compute $Z_n^{j, N, Sig} = Z_n^{\theta_n}(\pi_m(Sig(\tilde X_{0: n\Delta t}^{j,\tilde N})))$
        \item For each path $\lbrace \tilde X^{j,\tilde N}_{0:\tilde N} \rbrace_{1 \leq j \leq B}$, find $Y_n^{j,N, Sig}$ iteratively using the Euler scheme \eqref{bsde_alg2}. 
    \end{itemize}}
    \STATE{Compute the loss by finding the variance the batch $\lbrace Y^{j, N, Sig}_0\rbrace_{1 \leq j\leq B}$
    $$Loss(Y^{N, Sig}_0)=\text{Var}(Y^{N, Sig}_0)$$
    }
    \STATE{Update the parameters through gradient descent.}
\ENDWHILE
\RETURN $Y_0$
\end{algorithmic}
\end{algorithm}

\section{Proof of convergence}
\subsection{Assumptions}
\begin{assumption}
   Let $b, \sigma, f, g$ be deterministic functions such that: 
   \begin{enumerate}
       \item $b(\cdot, 0),\sigma(\cdot, 0), f(\cdot, 0, 0, 0)$ and $g(0)$ be uniformly bounded. 
       \item $b,\sigma, f, g$ are uniformly lipschitz continuous in $(x,y,z)$ with lipschitz constant $L$.
       \item $b,\sigma, f$ are uniformly Hölder-$\frac{1}{2}$ continuous in $t$ with Hölder constant $L$.
       \item $f$ has slow and at most linear growth in $y$ and $z$:
       \begin{align*}
           |f(t,x,y_1,z_1)-f(t,x,y_2,z_2)|^2 \leq K_y |y_1-y_2|^2+K_z|z_1-z_2|^2
       \end{align*}
       with $K_y$ and $K_z$ sufficiently small.
   \end{enumerate}
\end{assumption}
We comment that those assumptions are standard for the existence of the BSDE except for \textit{4} which is needed for the proof for the backward method. 

We want to identify the solution of the BSDE with the solution of a Semi-linear PDE according to the non-linear Feynman-Kac formula. Hence, we make the following assumption.
\begin{assumption}
Assume that the following PDE has a classical solution $u \in C^{1,2}([0,T) \times \bR^{d_1}; \bR^{d_2})$
    \begin{align}
		\begin{cases}
			\partial_t u(t,x) + \mathcal{L}u(t,x) + f(t,x,u,\sigma^T\partial_x u)=0, \ (t,x)\in [0,T) \times \bR^d \\ 
			u(T,x) =g(x) , \ x \in \bR^d	
		\end{cases}
	\end{align}
\end{assumption}
Then under such assumption, one can write
\begin{align*}
    Y_t=u(t,X_t), \ \ Z_t = \sigma^T(X_t) D_x u(t,X_t)
\end{align*}

In all the analysis that follows, we will take $d=d_1=d_2=1$ for simplicity. 
We introduce the standing assumptions below. 

\subsection{Markovian Case}
Some simple estimates are given below 
For simplicity, we use the following short hand notations, and $C$ denotes a generic constant which may differ from line to line. 
\begin{enumerate}[i.]
    \item $\Delta f_n = f(t , X_{t}, Y_{t}, Z_{t}) - f_n(X_{t_n}^{N},Y^{N,Sig}_{t_n}, Z^{N,Sig}_{t_n})$; $ \ \ \Delta f^N_n = f(t, X_{t}, Y_{t}, Z_{t}) - f_n(X_{t_n},Y_{t_n}, Z_{t_n})$; 
    \item $\Delta X_n= X_{t}-X_{t_n}^{N}$; $\ \ \Delta X^N_n= X_{t}-X_{t_n}$
    \item $\Delta Y_n= Y_{t}-Y^{N,Sig}_{t_n}$; $\ \ \Delta Y^N_n= Y_{t}-Y_{t_n}$
    \item $\Delta Z_n= Z_{t}-Z^{N,Sig}_{t_n}$; $\ \ \Delta Z^N_n= Z_{t}-Z_{t_n}$
    \item $\Delta g_n=g(X_{t})-g(X_{t_n}^{N})$; $\ \ \Delta g^N_n=g(X_{t})-g(X_{t_n})$
\end{enumerate}

The following standard result from \cite{Jianfeng} is needed for the proof later
\begin{theorem}\label{standard1}
    Let \textit{Assumption 1} holds and assume $h$ is small. Then
    \begin{align}
        \max_{0 \leq n \leq N} \E[ \sup_{t_n \leq t \leq t_{n+1}}|Y_t-Y_{t_n}|^2] +\sum^{N-1}_{n=0} \E[\int^{t_{n+1}}_{t_n} |Z_t - Z_{t_n}|^2 dt ] \leq C (1+|x|^2) h 
    \end{align}
\end{theorem}

\begin{lemma}
    By \textit{Assumption 1} and \textit{Assumption 2}, the following inequality holds true
    \begin{enumerate}[i.]
        \item $\E[|\Delta X_n|^2] \leq C \Delta t $; $\ \ \E[|\Delta X^N_n|^2] \leq C \Delta t $
        \item $\E[\int_{t_n}^{t_{n+1}}|\Delta f_n |^2 dt] \leq C \Delta t^2 +C \E[|Y_{t_n}-Y^{N,Sig}_{t_n}|^2 \Delta t]+ C\E[\int_{t_n}^{t_{n+1}}|\Delta Z_n|^2 dt]$ 
    \end{enumerate}
\end{lemma}
\begin{proof}
    The result for i) is standard which is from Theorem 5.3.1 in \cite{Jianfeng}. ii) follows from the following sequence of inequalities and Theorem \ref{standard1}: for $t_n \leq t \leq {t_{n+1}}$ for any $t$:
        \begin{align*}
            \E[\int_{t_n}^{t_{n+1}}|\Delta f_n |^2 dt]& \leq 2\E[\int_{t_n}^{t_{n+1}}|f(t,X_t,Y_t,Z_t)-f(t_n,X_{t_n},Y_{t_n},Z_{t})|^2 \nonumber \\ 
            &+|f(t_n,X_{t_n},Y_{t_n},Z_{t})-f(t_n,X^N_{t_n},Y^{N,Sig}_{t_n}, Z^{N,Sig}_{t_n}) |^2 dt]\\
            &\leq C \Delta t^2 + C \E[\int_{t_n}^{t_{n+1}}|\Delta Y^N_n|^2  +|Y_{t_n}-Y^{N,Sig}_{t_n}|^2 dt]+ C\E[\int_{t_n}^{t_{n+1}}|\Delta Z_n|^2 dt] \\
            & \leq C \Delta t^2 +C \E[|Y_{t_n}-Y^{N,Sig}_{t_n}|^2 \Delta t]+ C\E[\int_{t_n}^{t_{n+1}}|\Delta Z_n|^2 dt]
        \end{align*}
\end{proof}

\subsubsection{Forward Algorithm}
\begin{lemma}\label{forward_z_lemma}
    Under \textit{Assumption 1} and \textit{Assumption 2}, further assume that $Z_t=\sigma^T(X_t)D_xu(t,X_t)$ is Lipschitz uniformly in $t$, then there exist $Z^{N,Sig}_{t_n}(\cdot)$ for different $0\leq n\leq N-1$ such that the following is true
    \begin{align}
        \sum^{N-1}_{n=0} \E[\int^{t_{n+1}}_{t_n}|Z_t-Z^{N,Sig}_{t_n}|^2 dt] \leq C \Delta t
    \end{align}
\end{lemma}
\begin{proof}
    Note that we have the following sequence of inequalities
    \begin{align}
        \sum^{N-1}_{n=0}\E[\int^{t_{n+1}}_{t_n}|Z_t-Z^{N,Sig}_{t_n}|^2 dt]&\leq 2 \sum^{N-1}_{n=0}\E[\int^{t_{n+1}}_{t_n}|Z_t-Z_{t_n}|^2+|Z_{t_n}-Z^{N,Sig}_{t_n}|^2 dt] \nonumber \\ 
        & \leq C\Delta t+ 4\sum^{N-1}_{n=0}\E[\int^{t_{n+1}}_{t_n}|Z_{t_n}(X_{t_n})-Z_{t_n}(X^N_{t_n})|^2+|Z_{t_n}(X^N_{t_n})-Z^{N,Sig}_{t_n}|^2 dt] \nonumber \\
        & \leq C\Delta t + 4\sum^{N-1}_{n=0} \E[\int^{t_{n+1}}_{t_n}|Z_{t_n}(X^N_{t_n})-Z^{N,Sig}_{t_n}|^2 dt] \nonumber \\ 
        & \leq C\Delta t
    \end{align}
    The last term can be made arbitrarily small by the universal non-linearity of the path signature. 
\end{proof}
Next, we state and prove the main theorem of the algorithm. 
\begin{theorem}\label{forward_main_theorem_markovian}
    Under \textit{Assumption 1} and \textit{Assumption 2} then there exist $Z^{N,Sig}_n(\cdot)$ for different $0\leq n\leq N-1$ such that the following is true
    \begin{align}
        \max_{0\leq n \leq N} \E[\sup_{t_n \leq t \leq t_{n+1}}|Y_t-Y^{N,Sig}_{t_n}|^2] \leq C \Delta t
    \end{align}
\end{theorem}
\begin{proof}
Take the difference between the following set of equations 
    \begin{align} \label{fbsde_diff}
	\begin{cases}
		Y_{t_{n+1}}=Y_{t_n}-\int^{t_{n+1}}_{t_n} f(t,X_t,Y_t, Z_t) dt + \int^{t_{n+1}}_{t_n} Z_t dW_t\\
		Y^{N,Sig}_{t_{n+1}}=Y^{N,Sig}_{t_n}-\int^{t_{n+1}}_{t_n} f(t_n,X^N_{t_n},Y^{N,Sig}_{t_n}, Z^{N,Sig}_{t_n}) dt + \int^{t_{n+1}}_{t_n} Z^{N,Sig}_{t_n}dW_t
	\end{cases}
\end{align}  
and we obtain the following
\begin{align}
    \Delta Y_{n+1} =  \Delta Y_n -\int^{t_{n+1}}_{t_n} (\Delta f_n) dt + \int^{t_{n+1}}_{t_n} (Z_t-Z^{N,Sig}_n) dW_t
\end{align}
Denote $\Delta \bar{Y}^{N,Sig}_n:=Y_{t_{n}}-Y^{N, Sig}_{t_{n}}$
By squaring both sides and using Young's inequality with epsilon, we have 
\begin{align}
    \E | \Delta \bar{Y}^{N,Sig}_{n+1}|^2 &\leq (1+\frac{\Delta t}{\epsilon})\E[\E[|\Delta \bar{Y}^{N,Sig}_n+ \int^{t_{n+1}}_{t_n} (Z_t-Z^{N,Sig}_n) dW_t|^2 | \calf_{t_n}] ] + (1+\frac{\epsilon}{\Delta t})\E[\Big( -\int^{t_{n+1}}_{t_n} \Delta f_n dt  \Big)^2] \nonumber \\ 
    & \leq (1+\frac{\Delta t}{\epsilon})\E[\E[|\Delta \bar{Y}^{N,Sig}_n|^2 + |\int^{t_{n+1}}_{t_n} (Z_t-Z^{N,Sig}_n) dW_t|^2 | \calf_{t_n}]] + 2(1+\frac{\epsilon}{\Delta t}) \Big( \E[\int^{t_{n+1}}_{t_n} |\Delta f_n|^2 dt ] \Delta t  \Big)\nonumber \\ 
    &\leq (1+\frac{\Delta t}{\epsilon})\E[|\Delta \bar{Y}^{N,Sig}_n|^2]+ C(\Delta t^2+ \E[\int^{t_{n+1}}_{t_n} |Z_t-Z^{N,Sig}_n|^2 dt] +\E[|\Delta \bar{Y}^{N,Sig}_n|^2] \Delta t ) \nonumber\\
    &+ C\E[\int^{t_{n+1}}_{t_n} |Z_t-Z^{N,Sig}_n|^2] dt \nonumber \\
    & \leq (1+\frac{\Delta t}{\epsilon})\E[|\Delta \bar{Y}^{N,Sig}_n|^2]+ C(\Delta t^2+  \E[\int^{t_{n+1}}_{t_n} |Z_t-Z^{N,Sig}_n|^2] dt+\E[\int^{t_{n+1}}_{t_n} |\Delta Z^N_n|^2] dt) \nonumber
\end{align}
Then, by Discrete Gronwall inequality, by Lemma \ref{forward_z_lemma} and Theorem \ref{standard1}, we have 
\begin{align}
    \sup_{0 \leq n \leq N-1} \E[|Y_{t_n}-Y^{N,Sig}_{t_n}|^2] \leq C \Delta t
\end{align}
Then together with Lemma \ref{standard1}, the claim is proved. 
\end{proof}

\subsubsection{Backward Algorithm}
Since the convergence takes the variance of $Y^{N,Sig}_0$ as the loss function, we first prove an a-priori estimate for the loss. 

Now we state the main thereom regarding the convergence of the backward scheme. 
We also provide a lower bound for $\text{Var}(Y^{N,Sig}_0)$
\begin{theorem}
For any $\epsilon >0$, we have 
     \begin{align} \label{var_Y_lower}
        \text{Var}(Y^{N,Sig}_0) > (1-\epsilon -\frac{2T K_z}{\epsilon}) \E[\int^T_{0} |\Delta Z_n|^2 dt] -C\Delta t -\frac{4T^2K_y}{\epsilon}\E[ \max_{0 \leq n \leq N-1} |Y_{t_n}-Y^N_{t_n}|^2 ]
    \end{align}
\end{theorem}
\begin{proof}
  Note that the true solution $Y_0$ is deterministic and that $\text{Var}(Y^{N,Sig}_0)= \text{Var}(Y_0-Y^{N,Sig}_0)$, 
    we have the following inequalities: 
    \begin{align}
        \text{Var}(Y^{N,Sig}_0) &= \E \Big[
       \big(\Delta g_n - \int^T_0 \Delta f_n ds + \int^T_0 \Delta Z_n dW_t - \E[\Delta g_n -\int^T_0\Delta f_n ds ] \big)^2
       \Big] \nonumber\\ 
       & \geq \E[\int^T_{0} |\Delta Z_n|^2 dt]+2  \E \Big[\int^T_t \Delta Z_n dW_t(\Delta g_n-\int^T_0 \Delta f_n dt -\E[\Delta g_n-\int^T_0 \Delta f_n dt]) \Big] \nonumber \\
      % & + \E[(\Delta g_n-\int^T_0 \Delta f_n dt-\E[\Delta g_n-\int^T_0 \Delta f_n dt] )^2] \nonumber \\ 
       & \geq (1-\epsilon) \E[\int^T_{0} |\Delta Z_n|^2 dt] - \frac{1}{\epsilon}\E[(\Delta g_n-\int^T_0 \Delta f_n dt -\E[\Delta g_n+\int^T_0 \Delta f_n dt])^2] \\ 
       & \geq (1-\epsilon) \E[\int^T_{0} |\Delta Z_n|^2 dt] - \frac{1}{\epsilon}\E[(\Delta g_n-\int^T_0 \Delta f_n dt)^2] \nonumber \\ 
       & \geq (1-\epsilon) \E[\int^T_{0} |\Delta Z_n|^2 dt] - \frac{2}{\epsilon} ( \E[ |\Delta g_n|^2]+ T \E[\int^T_0 |\Delta f_n|^2 dt] ) \nonumber \\ 
       & \geq (1-\epsilon) \E[\int^T_{0} |\Delta Z_n|^2 dt] - C\Delta t -\frac{2}{\epsilon}(C\Delta t+T\E[\int^T_0 K_y |\Delta Y_n|^2 + K_z|\Delta Z_n|^2 dt]) \nonumber \\ 
       & \geq (1-\epsilon -\frac{2T K_z}{\epsilon}) \E[\int^T_{0} |\Delta Z_n|^2 dt] -C\Delta t -\frac{2TK_y}{\epsilon}\E[\int^T_0 |Y_t -Y_{t_n}+Y_{t_n}-Y^N_{t_n}|^2dt ] \nonumber \\
       & \geq (1-\epsilon -\frac{2T K_z}{\epsilon}) \E[\int^T_{0} |\Delta Z_n|^2 dt] -C\Delta t -\frac{4T^2K_y}{\epsilon}\E[ \max_{0 \leq n \leq N-1} |Y_{t_n}-Y^N_{t_n}|^2 ] \label{lower_var}
    \end{align}
    where in the inequalities we used the fact that $Var(X) \leq \E[X^2]$ and the second inequality that $2 ab \geq -\epsilon a^2 -\frac{1}{\epsilon^2} b^2$. In the second to last inequality, we used Theorem \ref{standard1}.
\end{proof}
We state the following main theorem about the convergence of the backward algorithm. 
\begin{theorem} \label{theorem_yz}
    Under \textit{Assumption 1} and \textit{Assumption 2}, the following inequality holds 
    \begin{align}
        \sup_{t \in[0,T]} \E[|Y_t - Y^{N, Sig}_{t_n}|^2]+ \int^T_0 \E[|Z_t-Z^{N, Sig}_{t_n}|^2 dt] \leq C(\Delta t+\text{Var}(Y^{N,Sig}_0))
    \end{align}
\end{theorem}
\begin{proof}
    Take the difference between the exact solution and the numerical solution, we have the following inequality: 
    \begin{align}
        \E[|Y_{t_n}-Y^{N,Sig}_{t_n}|] &\leq 3 \Big(  \E[|\Delta g|^2]+ T \E[\int^T_0 |\Delta f_n|^2 dt  + \int^T_0 |\Delta Z_n|^2 dt] \Big) \nonumber\\ 
        &\leq 3 C\Delta t+ C\Delta t+ 3T\E[\int^T_0 K_y |\Delta Y_n|^2 + K_z |\Delta Z_n|^2 dt] + 3 \E[\int^T_0 |\Delta Z_n|^2 dt ] \nonumber \\ 
        &\leq C\Delta t + 3(1+TK_y)\E[\int^T_0 |\Delta Z_n|^2 dt]+ 3TK_y\E[\int^T_0 |Y_t -Y_{t_n}+Y_{t_n}-Y^{N,Sig}_{t_n}|^2dt] \nonumber \\ 
        & \leq C\Delta t+3(1+TK_y)\E[\int^T_0 |\Delta Z_n|^2 dt]+ 6T^2K_y \max_{0 \leq n \leq N-1} \E[|Y_{t_n}-Y^{N,Sig}_{t_n}|^2]
    \end{align} 
This shows that 
\begin{align}
     \max_{0 \leq n \leq N-1}\E[|Y_{t_n}-Y^{N,Sig}_{t_n}|^2] &\leq \frac{C \Delta t + 3(1+TK_y)\E[\int^T_0 |\Delta Z_n|^2 dt]}{1-6T^2 K_y} \nonumber \\ 
     &\leq C\Delta t+  \frac{3(1+TK_y)}{1-6T^2 K_y}\E[\int^T_0 |\Delta Z_n|^2 dt] \label{y_bound}
\end{align}
Using Theorem 4.4 (\eqref{lower_var}), we have 
\begin{align}
    \text{Var}(Y^{N,Sig}_0) &\geq (1-\epsilon -\frac{2T K_z}{\epsilon}) \E[\int^T_{0} |\Delta Z_n|^2 dt] -C\Delta t -\frac{4T^2K_y}{\epsilon} \big( C\Delta t+  \frac{3(1+TK_y)}{1-6T^2 K_y}\E[\int^T_0 |\Delta Z_n|^2 dt] \big) \nonumber \\
    & \geq \Big( (1-\epsilon -\frac{2T K_z}{\epsilon})+ \frac{4T^2K_y}{\epsilon}\frac{3(1+TK_y)}{1-6T^2 K_y} \Big)\E[\int^T_0 |\Delta Z_n|^2 dt]-C\Delta t
\end{align}
Then take $\epsilon= \sqrt{2T K_z+ 4T^2K_y \frac{3(1+TK_y)}{1-6T^2 K_y}}$ by the assumption on $K_z$ and $K_y$, since $2 \epsilon <1$
This implies that 
\begin{align}
    \E[\int^T_0 |\Delta Z_n|^2 dt] \leq C(\Delta t+\text{Var}(Y^{N,Sig}_0)) 
\end{align}
together with \eqref{y_bound} and Theorem \ref{standard1}, we have 
    \begin{align}
        \sup_{t \in[0,T]} \E[|Y_t - Y^{N, Sig}_{t_n}|^2]+ \int^T_0 \E[|Z_t-Z^{N, Sig}_{t_n}|^2 dt] \leq C(\Delta t+\text{Var}(Y^{N,Sig}_0))
    \end{align}
\end{proof}

\begin{theorem}\label{var_upper}
    Under \textit{Assumption 1} and \textit{Assumption 2}, the following inequality holds 
    \begin{align} \label{var_Y_upper}
        \text{Var}(Y^{N,Sig}_0) \leq C \Big(\Delta t+\sum_{0 \leq n \leq N-1} \E[|Z_{t_n}(X^N_{n})- Z^{N,Sig}_{t_n}|^2] \Delta t\Big)
    \end{align}
\end{theorem}
\begin{proof}
    Again, note that the true solution $Y_0$ is deterministic and that $\text{Var}(Y^{N,Sig}_0)= \text{Var}(Y_0-Y^{N,Sig}_0)$, 
    we have the following result: 
    \begin{align*}
       \text{Var}(Y^{N,Sig}_0) &= \E \Big[
       \big(\Delta g_n - \int^T_0 \Delta f_n ds + \int^T_0 \Delta Z_n dW_t - \E[\Delta g_n -\int^T_0\Delta f_n ds ] \big)^2
       \Big] \\ 
       & \leq 2 \E \Big[  (\Delta g_n -\int^T_0 \Delta f_n ds- \E[\Delta g_n +\int^T_0\Delta f_n ds ])^2 + \int^T_0 |\Delta Z_n|^2 dt \Big] \\ 
       & \leq2 \E \Big[  (\Delta g_n - \int^T_0 \Delta f_n ds)^2 + \int^T_0 |\Delta Z_n|^2 dt \Big] \\ 
       & \leq 4 \E \Big[  |\Delta g_n|^2  + T \int^T_0 \Delta |f_n|^2 ds] + 2 E[\int^T_0 |\Delta Z_n|^2 dt \Big]\\ 
       & \leq C \Delta t + C\E[\int^T_0 |\Delta Y_n|^2 + |\Delta Z_n|^2  dt]+ 2 E[\int^T_0 |\Delta Z_n|^2 dt \Big] \\ 
       & \leq C\Delta t + C \max_{0 \leq n \leq N-1}\E |\Delta Y_n|^2+CE[\int^T_0 |\Delta Z_n|^2 dt \Big] 
    \end{align*}
where we used \eqref{y_bound} to bound $\max_{0 \leq n \leq N-1}\E[|\Delta Y_n|^2]$.
Note that by Theorem \ref{standard1}, one has that 
\begin{align}
    E[\int^T_0 |\Delta Z_n|^2 dt \Big] & \leq 2 E[\int^T_0 |Z_{t}-Z_{t_n} |^2 +|Z_{t_n}-Z^{N,Sig}_{t_n}|^2 dt \Big] \nonumber \\ 
    &\leq 4 \E[\int^T_0 |Z_{t}-Z_{t_n} |^2 +|Z_{t_n}(X^{N}_{t_n})-Z_{t_n}(X_{t_n})|^2 + |Z_{t_n}(X^{N}_{t_n})-Z^{N,Sig}_{t_n}|^2 dt \Big] \nonumber \\ 
    & \leq C\Delta t + \sum_{0 \leq n \leq N-1} \E[|Z_{t_n}(X^{N}_{t_n})-Z^{N,Sig}_{t_n}|^2] \Delta t \label{deltaZ_est}
\end{align}
Note that $Z_{t_n}(X^{N}_{t_n})$ can be approximated by $Z^{N,Sig}_{t_n}$ arbitrarily close given the same underlying path $\lbrace X^{N}_{t_n} \rbrace_{1 \leq n \leq N-1}$ due to the universal non-linearity of signature path. As such, inequality \eqref{var_Y_upper} holds. 
\end{proof}

By putting together Theorem \ref{var_upper} and Theorem \ref{theorem_yz} we have the following result, which then shows convergence by universality property of the path signature.
\begin{theorem} \label{backward_theorem_vf}
     Under \textit{Assumption 1} and \textit{Assumption 2}, the following inequality holds 
    \begin{align}
        \sup_{t \in[0,T]} \E[|Y_t - Y^{N, Sig}_{t_n}|^2]+ \int^T_0 \E[|Z_t-Z^{N, Sig}_{t_n}|^2 dt] \leq C(\Delta t+\sum_{0 \leq n \leq N-1} \E[|Z_{t_n}(X^N_{n})- Z^{N,Sig}_{t_n}|^2] \Delta t)
    \end{align}
\end{theorem}

\subsection{Non-Markovian Generalization}
We have provided the proof for the Markovian case. We now consider the case for Non-Markovian FBSDE, i.e. both the driver and the terminal function $g$ depend on the entire path of the process $X_{[0,t]}$. For simplicity we still consider dimension $d=1$. We introduce the following notation for the path analysis which are mainly from Dupire \cite{dupire}. 

For each $\omega \in \Omega$, $X: \Gamma \rightarrow \bR$ is the canonical process: $X_t(\omega):=\omega_t$. 
Let $|\cdot|$ denote the norm. For each $0\leq t \leq t' \leq T$, define $\Lambda:= \underset{0 \leq t \leq T} \bigcup \Lambda_t$ where $\Lambda_t: [0,t] \rightarrow \bR$ is the path. We write: 
\begin{align*}
    ||\omega_t||&:= \sup_{r \in [0,t]}|\omega(t)| \\ 
    d_{\infty}(\omega_t, \omega'_{t'})&:=\max ( \sup_{r \in [0,t)} \lbrace |\omega(r)-\omega'(r)|  \rbrace, \sup_{r \in [t,t']} \lbrace |\omega(t)-\omega'(r)|  \rbrace )+|t-t'|
\end{align*}
In Dupire's formulation of derivatives one often regard $u(\omega_t)$ as a function of $t, \omega, x$
\begin{align}
    u(\omega^x_t):=u(t,\omega(s)_{0 \leq s < t}, \omega(t)+x)
\end{align}
Definition of spatial derivative is given below
\begin{definition}
    Let $u: \Lambda \rightarrow \bR$ and $\omega_t \in \Lambda$, if there exists $p \in \bR^d$ such that 
    \begin{align}
        u(\omega^x_t)=u(\omega_t)+p\cdot x+o(|x|), \ \ x \in \bR
    \end{align}
then we say that $u$ is vertically differentiable at $\omega_t$ and denote the gradient of $D_x u(\omega_t)=p$. $u$ is said to be vertically differentiable in $\Lambda$ if $D_x u(\omega_t)$ exists for each $\omega_t \in \Lambda$. We also define the Hessian $D_{xx}u(\omega_t)$. In a similar fashion, the hessian is an $\mathbf{S}(d)$-valued function defined on $\Lambda$, where $\mathbf{S}(d)$ is the space of $d\times d$ symmetric matrices. 
\end{definition}
For the horizontal derivative, we first define $\omega_t \in \Lambda$, we denote 
\begin{align}
    \omega_{t,s}(r)=\omega(r)\mathbf{1}_{[0,t)}(r)+ \omega(t)\mathbf{1}_{[t,s]}(r), \ \ r \in [0,s]
\end{align}
\begin{definition}
    For a given $\omega_t \in \Lambda$ if we have 
    \begin{align}
        u(\omega_{t,s}) = u(\omega_t) + a (s-t) + o(|s-t|), \ \ s\geq t
    \end{align}
    then we say that $u(\omega_t)$ is horizontally differentiable in $t$ at $\omega_t$ and denote $D_tu(\omega_t)=a$. $u$ is said to be horizontally differentiable in $\Lambda$ if $D_tu(\omega_t)=a$ exists for all $\omega_t \in \Lambda$. 
\end{definition}
By defining $u(t,X_t):=Y_t$ we have the nonlinear Feynman-Kac formula from Proposition 3.8 \cite{peng} that by denoting $X_t$ as the path of $X_{. \wedge t}$ and $x_t$ is the spatial value at time $t$.  
\begin{align}
    Z_t=D_x u (t,X_t) \sigma(t,x_t) := F(t, X_t)
\end{align}
We make the following strong assumption on $u(t,X_t)$ and $F(t,X_t)$: 
\begin{assumption}
    Both $u$ and $F$ are  Hölder-$\frac{1}{2}$ continuous in $t$ and they are uniformly Lipschitz continuous in $X$ with Lipschitz constant L.
    Namely, for $f= u , F$
    \begin{align}
        |f (t,X_t)-f(t',X_{t})| &\leq C |t-t'|^{\frac{1}{2}} \\
        |f (t,X_t)-f(t,X'_{t'})| &\leq L d_{\infty}(X_t,X'_{t'})
    \end{align}
In addition to the Assumption 1, assume the driver of the BSDE is also Hölder-$\frac{1}{2}$ continuous in $t$ and it is uniformly Lipschitz continuous in $X$ with Lipschitz constant L in the sense above.
\end{assumption}

One immediate consequence of this assumption is that Theorem \ref{standard1} holds also in the non-markovian setting. 
\begin{align}
    \int^{t_{n+1}}_{t_n}\E[|Z_t-Z_{t_n}|^2]dt &=\int^{t_{n+1}}_{t_n}\E[|F(t,X_t)-F(t_n,X_{t_n})|^2]dt  \nonumber \\ 
    & \leq \int^{t_{n+1}}_{t_n} C(\Delta t + d^2_{\infty}(X_t,X_{t_n}) ) dt \nonumber \\ 
    & \leq \int^{t_{n+1}}_{t_n} C(\Delta t + \E[\sup_{t_n \leq t \leq t_{n+1}}|X_{t}-X_{t_n}|^2 ]) dt \\ 
    & \leq \int^{t_{n+1}}_{t_n} C \Delta t dt
\end{align}
This will imply that 
$$ \sum^{N-1}_{n=0} \E[\int^{t_{n+1}}_{t_n} |Z_t - Z_{t_n}|^2 dt ] \leq C \Delta t $$ 
In a similar fashion one can show that 
$$ \max_{0 \leq n \leq N} \E[ \sup_{t_n \leq t \leq t_{n+1}}|Y_t-Y_{t_n}|^2] \leq C \Delta t$$
This shows that the standard Theorem also holds for the path dependent case. 

As a result, by the same argument as in the proof for Lemma \ref{forward_z_lemma} the Non-Markovian version of the  Lemma holds. 

For a non-Markovian version of the theorem as Theorem \ref{forward_main_theorem_markovian}, we note that for the driver of the BSDE, the randomness in the driver is introduced through the entire path of $X_t$ instead of just the spatial value $x_t$. We study the difference between the dependency of the driver $f(t,\cdot, Y, Z)$ on $X_t$ which is the exact dynamics (path) of the SDE and the discrete approximation of the path: $ \lbrace X^N_{t_n} \rbrace_{n=0,..,N}$ which is a piecewise constant function. For $t \in [t_n,t_{n+1}]$
\begin{align}
    \E[|f(t,X_t,Y_t,Z_t)-f(t,X^N_{t_n},Y_t,Z_t)|^2] & \leq \E[d^2_{\infty}(X_t,X_{t_n}) ] \nonumber \\ 
    & \leq \E[\sup_{t_n \leq t \leq t_{n+1}}|X_{t}-X_{t_n}|^2] \nonumber \\ 
    & \leq C\Delta t 
\end{align}
As such, by using similar argument, one can show that the non-Markovian version of Theorem \ref{forward_main_theorem_markovian} and Theorem \ref{backward_theorem_vf} also hold. 

\section{Numerical Examples}
In this section, we provide numerical examples and compare the results obtained using our algorithms to those from the literature \cite{qifeng1}, \cite{qifeng2} and \cite{come1}. For all the numerical examples, we generate new batch samples (data) for each iteration. This is different from examples from \cite{qifeng1} in the reference, where a fixed amount data was pre-generated, say $10^5$ trajectories. The main motivation for us to use new samples for each iteration is 1) we want to avoid over fitting. 2) the data generation process is independent from the training process and it is easy to implement. We comment that the we will generate the path signature for $\lbrace X_{\tilde t_j} \rbrace^B_{0  \leq \tilde t_j \leq  t_n}$ for each $n=0,1,...N-1$, and $l \in B$ where $B$ is the batch. And one can use Chen's identity for the path generation to avoid regeneration of the signatures for the overlapping parts in the paths. 

\subsection{Lookback Options}
Consider the Blackscholes seting where the stock price follows the following dynamics
\begin{align}
    dX_t=r X_t dt + \sigma X_t dW_t 
\end{align}
The terminal payoff is defined to be 
\begin{align}
    g(X_{\cdot \wedge T}) = X_T -\inf_{0\leq t \leq T} X_t
\end{align}
The option price is defined to be 
\begin{align}
    Y_t=\exp(-r(T-t)) \E[ g(X_{\cdot \wedge T})|\calf_t]
\end{align}
Then $e^{-rt}Y_t=\E[ e^{-rT}g(X_{\cdot \wedge T})|\calf_t] $ is a martingale. Meaning that 
\begin{align}
    d(e^{-rt}Y_t)= Z'_t dW_t
\end{align}
for some square integrable process $Z'_t$, then one immediately has 
\begin{align}
    dY_t = rY_t dt + Z_t dW_t
\end{align}
where $Z_t:= e^{rt} Z'_t$. 
The solution $Y_t$ has the following analytic formula
\begin{align}
    Y_t=X_t\Phi(p_1)-m_t e^{-r(T-t)} \Phi(p_2) -X_t \frac{\sigma^2}{2r}\big( \Phi(-p_1)  -e^{-r(T-t)}(\frac{m_t}{yt})^{2r/\sigma^2} \Phi(-p_3) \big)
\end{align}
where 
$$m_t:=\inf_{0\leq u\leq t} X_t, p_1=\frac{\log(X_t/m_t)+(r+\sigma^2/2)(T-t)}{\sigma\sqrt{T-t}}, p_2=p_1-\sigma \sqrt{T-t}, p_3=p_1-\frac{2r}{\sigma}(\sqrt{T-t})$$

For the model specification, similar to \cite{qifeng1}, we use $\tilde N = 2000$ and $N=20$. We use the truncated signature method with level $m=3$. 
We take $x_0=1, \sigma=1, r=0.01, T=1$. Since in \cite{qifeng1} it is demonstrated that the method therein is state of the art, we present comparison between results obtained using methods in \cite{qifeng1} and our forward/backward methods in table \ref{eg1_table}. 

For the forward algorithm (method 1), we take batch size to be 100. And we use the Adam optimizer with learning rate $10^{-3}$. since $N=20$, we use 20 individual neural network to approximate $Z_{t_n}$ for each $0\leq n \leq N-1$. We use fully connected feedforward neural networks with 2 hidden layers each of 64 neurons. 
It can be observed from Figure \ref{fig: eg1_m1} that our forward algorithm produces good result compared to the exact solution: the blue curve is very close to the red dashline. 

For the backward algorithm (method), we note that due to the designed model methodology, it is desirable to use a larger batch size: we are minimizing the variance of the mini-batches and use the mean of the batch samples as the sample estimate for $Y_0$. We observe that for this method, convergence happens  fast but with values typically fluctuating about a fixed level. This is due to the fact that 
\begin{itemize}
    \item We are not training the model on a fixed amount of trajectories of $\lbrace X_{\tilde t_j} \rbrace^M_{0 \leq j \leq \tilde N}$ where M is a fixed amount (say 10000). Instead, to overcome over-fitting, we generate different batch samples for each iteration. Hence, one trajectory of sample path may leads to high variance. 
    \item  We minimize the variance of the sample and use the mean of the sample as estimation for $Y_0$, so the estimated mean may fluctuate according to sample size. 
\end{itemize}
For this numerical example, we use $B=1000$ as the batch sample size. We perform 50 runs of the algorithm and obtain mean of $5.78$. The confidence interval is $[5.77,5.79]$. As such, even though the estimated $Y_0$ for each iteration shows some fluctuation in the sample run, the mean is stable.  We show the result of one the sample run in Figure \ref{fig: eg1_m2}. 

As such, we observe that both methods are comparable to the existing in \cite{qifeng1}.

\begin{figure}[!htb]
\center
    \includegraphics[width=0.6\textwidth]{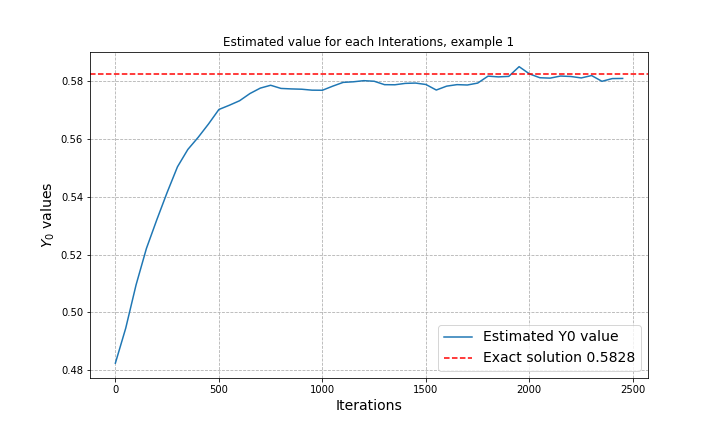}
    \caption{Predicted $Y_0$ value versus the number of interations trained.}
    \label{fig: eg1_m1}
\end{figure}

\begin{figure}[!htb]
\center
    \includegraphics[width=0.6\textwidth]{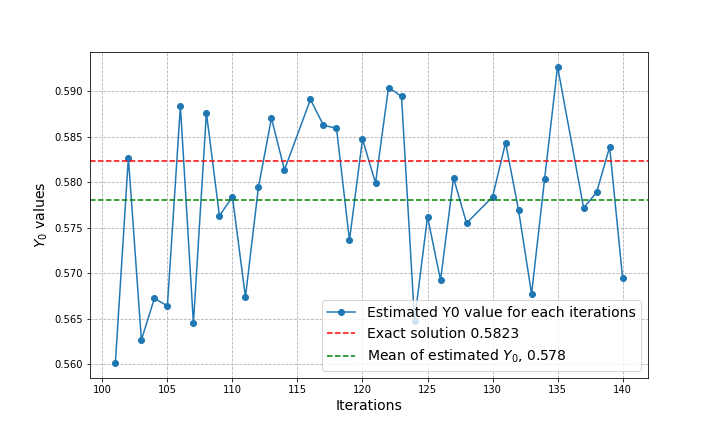}
    \caption{Predicted $Y_0$ value versus the number of interations trained.}
    \label{fig: eg1_m2}
\end{figure}

\begin{table}
\caption{Example 1 results comparison}	
\label{eg1_table}
\begin{center}
\begin{tabular}{ |c|c|c|c|c|c|} 
\hline
 & Exact & Method in \cite{qifeng1}& Forward method 1 & Backward method 2   \\
\hline
Result $Y_0$ & 5.828  & 5.79 & 5.81  & 5.78  \\
\hline
Error & -- & 0.6 \%& 0.3 \% & 0.7\% \\
\hline
\end{tabular}
\end{center}
\end{table}

\subsection{A Higher dimension example}
The last example has dimension $d=1$, for this particular example, we take $d=20$. We consider the following high dimension problem 
\begin{align}
\label{fbsde_eg2}
	\begin{cases}
		\int^T_t dX_s = \int^T_t dW_s \\
		Y_t=g(X_{\cdot \wedge T})+ \int^T_t f(s, X_{\cdot \wedge s}, Y_s, Z_s)ds -\int^T_t Z_s dW_s \\ 
	\end{cases}
\end{align}
For simplicity, we take $f=0$ and $g(X_{\cdot \wedge T})= (\int^T_0 \sum^{d}_i X^i_s ds)^2 $. 
We comment that because the dimension of the signature grows exponentially as function of the dimension of the underlying state process, we will use log-signature of paths instead of signatures since it essentially contains the same amount of information but requires less dimensions. We take $d=20, T=1, N=5, \tilde N=100$. 
Again, it is observed that the results of our forward method 1 and backward method 2 are comparable to the Exact solution and the method proposed in \cite{qifeng1}. We comment that the convergence of the forward method can be slow in this case as it could be sensitive to the initial guess of $Y_0$. But determination of a rough range of $Y_0$ is easy: try running the algorithm with guess $Y^1_0$, if the estimated result is monotonically increasing during the training process, then stop the algorithm and make a larger guess $Y^2_0$. Keep increasing the initial guess if it is still increasing otherwise make a smaller guess. Using this methodology, we start training our algorithm with $Y_0=6.0$ and it stabilizes at 6.58. Method 2 converges rapidly but there is some fluctuation among each iteration within one run. Again, we use the mean estimator and we observe the result is 6.71. The test result for $d=20$ is summarized in Table \ref{eg2_table}. 
\begin{table}
\caption{Example 2 result comparison, $d=20$}	
\label{eg2_table}
\begin{center}
\begin{tabular}{ |c|c|c|c|c|c|} 
\hline
 $d=20$& Exact & Method in \cite{qifeng1}& Forward method 1 & Backward method 2   \\
\hline
Result $Y_0$ & 6.66  & 6.6 & 6.58  & 6.71  \\
\hline
Error & -- & 1 \%& 1.2 \% & 0.75\% \\
\hline
$d=100$ \\
\hline
Result $Y_0$ & 33.33  & -- & 33.15  & 33.50  \\
\hline
Error & -- & NA & 0.538 \% & 0.534\% \\
\hline
\end{tabular}
\end{center}
\end{table}
We also present the result when the dimension is very high $d=100$ in which case the method in \cite{qifeng1} is no longer practical. In this case, we construct a (trainable) embedding layer whose output data stream has state dimension 5. Since the dimension is reduced, we can take larger $\tilde N$ for the accuracy purpose in which case is selected to be 10.

We present the PDE related to \eqref{fbsde_eg2} and state its exact solution for completeness. For $(t,\omega) \in \big( [0,T] \times C([0,T],\bR^d) \big)$ 
\begin{align}
    \begin{cases}
        \partial_t u + \frac{1}{2}tr(\partial_{\omega \omega} u) + f(t, \omega, u \partial_{\omega} u)=0 \\
        u(T,\omega)=g(\omega), \ \ g(\omega)= (\int^T_0 \sum^d_i \omega^i_s ds)^2
    \end{cases}
\end{align}
The exact solution is 
\begin{align*}
    u(t,\omega)=(\int^T_0 \sum^d_i \omega^i_s ds)^2 + (\sum^d_i \omega^i_t)^2(T-t)^2+2(T-t)(\sum^d_i \omega^i_t)\int^t_0 \sum^d_i \omega^i_s ds + \frac{d}{3}(T-t)^3
\end{align*}

\subsection{An example of non-linear type: Amerasian Option}
As the last example, we provide an example used in \cite{qifeng2}. We note that in this case, our forward algorithm 1 is not applicable anymore. We will consider only the backward algorithm according to \eqref{reflected_num}. 

Amerasian option is considered under the Black-Scholes model that involves $d$ stocks $X_1,...X_d$ which follows the following SDEs: 
\begin{align}
    dX^i_t=r X^i_t dt + \sigma_t^idW^i_t, \ X^i_0=x^i_0, i=1,...,d
\end{align}
where the $W^i$ are assumed to be independent. The payoff of the basket Amerasian call option at  strike price of $K$ is defined as 
\begin{align}
    g(X_{\cdot \wedge T})= \Big( \sum^d_{i=1} \frac{w_i}{T} \int^T_0 X^i_t dt -K\Big)^+
\end{align}
where $w_i$ are defined to be the weights. We consider the price of the Bermudan option which is a type of American option where early excercise can only happen at prescribed dates. The model variables are taken to be $X^i_0=100, r=0.05, \sigma_i=0.15, \omega =\frac{1}{d}, T=1, K=100$. We take $\tilde N =1000, N=20$. 
We show the benchmark results against 1) The European price. 2) The method from \cite{qifeng2} and 3) method from \cite{come1}. 
On the analytic side for benchmark purposes, we comment that: 
\begin{enumerate}[i.]
    \item The American option price should be higher than the European price due to its flexibility of being able to exercise early. 
    \item By Jensen's inequality, using the current parameters
    \begin{align*}
        \E[e^{-rT} \Big( \sum^d_{i=1}\frac{1}{dT} \int^T_0 X^i_t dt -K\Big)^+] \geq e^{-rT}\Big( \E[\frac{1}{dT} \sum^d_{i=1}  \int^T_0 X^i_t dt -K]\Big)^+=2.42
    \end{align*}
    \item Again, we can show that the price of the Ameransian option price is \textit{decreasing}. By using the Jensen inequality we have 
    \begin{align}
        e^{-rT}\Big( \frac{1}{2dT} \sum^{2d}_{i=1}  \int^T_0 X^i_t dt -K\Big)^+ \leq \frac{1}{2} e^{-rT}\Big( \frac{1}{dT} \sum^{d}_{i=1}  \int^T_0 X^i_t dt -K\Big)^+ + \frac{1}{2}e^{-rT}\Big( \frac{1}{dT} \sum^{2d}_{i=d+1}  \int^T_0 X^i_t dt -K\Big)^+  \nonumber
    \end{align}
    Then taking the limit in $d$ we obtain that the price converges to 2.42 when $d$ goes to infinity.  
\end{enumerate}

We show our results in Table \ref{eg3_table1}. We run the algorithm for dimensions $d=1,5,10,20,100$ for 50 runs. The mean of the 50 estimates is taken as the estimator and the confidence intervals are computed based on those samples.

\begin{table}
\caption{Example 3 result comparison $d=1$}	
\label{eg3_table1}
\begin{center}
\begin{tabular}{ |c|c|c|c|c|c|} 
\hline
 d=1& European & Method in \cite{qifeng2} & Method in \cite{come1} & Our Backward method  \\
\hline
Result $Y_0$ & 4.732  & 4.963& 5.113  & 5.03 \\
\hline
Confidence Interval & -- & [4.896,5.03] & [5.009,5.217]  & [4.97, 5.10]\\
\hline
 d=5         \\
\hline
Result $Y_0$ & 3.078  & 3.190 & 3.335  &  3.11 \\
\hline
Confidence Interval & -- & [3.115, 3.266] & [3.207, 3.462]  & [3.08,3.155] \\
\hline
 d=10  \\
\hline
Result $Y_0$ & 2.701  & 2.914 & 3.142 & 2.76 \\
\hline
Confidence Interval & -- & [2.844,2.983] &[2.975,3.309]   &[2.734, 2.813] \\
\hline
 d=20\\
\hline
Result $Y_0$ & 2.51  & 3.093 & 3.095 & 2.61 \\
\hline
Confidence Interval & -- & [3.017,3.168] &[2.883,3.3308]   & [2.587, 2.627] \\
\hline
\end{tabular}
\end{center}
\end{table}

We also note that our results agree with the theoretical findings. That is, compared to the benchmarks, the calculated prices are all above the European prices. We comment that each run takes much less time than reported in \cite{qifeng2}; For example, when $d=5$ the runtime (one run) of our algorithm takes only 369.47s to achieve the estimate while in \cite{qifeng2}, 1927.55s is reported. 
In the high dimension case where $d=100$, the price is found to be $2.516$ with confidence interval $[2.506,2.526]$. We note that in this case, the method in \cite{qifeng2} fails due to the excessively large size of both the signature and path signatures in this case.

In the meantime, Observing the trend in the price predicted, we note that our predicted price agree with the trend better than \cite{qifeng2}: when $d=20$ their predicted price even increased compared to when $d=10$, which does not align with theoretical result. We comment that when using the embedding layers, one run for $d=20$ reduces to 165.0s and for $d=100$ one run takes about 320.0s.

\clearpage
\bibliographystyle{apacite}
%\newpage

\end{document}